\documentclass[12pt,oneside,american]{amsart}
\usepackage[T1]{fontenc}
\usepackage[latin9]{inputenc}
\usepackage{amsthm}
\usepackage{amssymb}
\usepackage{setspace}
\usepackage{esint}
\makeatletter
\numberwithin{equation}{section}
\numberwithin{figure}{section}
\theoremstyle{plain}
\newtheorem{thm}{\protect\theoremname}[section]
  \theoremstyle{plain}
  \newtheorem{lem}[thm]{\protect\lemmaname}
  \theoremstyle{plain}
  \newtheorem{cor}[thm]{\protect\corollaryname}
  \theoremstyle{definition}
  \newtheorem{defn}[thm]{\protect\definitionname}

\usepackage{fancyhdr}
\usepackage[us,12hr]{datetime} 

\newcommand{\GcDM}{\Gamma_{c}(\overline{\mathcal{D}},\mathfrak{M}(d,n))}
\newcommand{\GhDM}{\Gamma_{h}(\overline{\mathcal{D}},\mathfrak{M}(d,n))}

\newcommand{\MnC}{M_{n}(\mathbb{C})}
\newcommand{\MnCd}{M_{n}(\mathbb{C})^d}

\usepackage{babel}

\makeatother

\usepackage{babel}
  \providecommand{\corollaryname}{Corollary}
  \providecommand{\definitionname}{Definition}
  \providecommand{\lemmaname}{Lemma}
\providecommand{\theoremname}{Theorem}

\begin{document}

\title{Boundaries, Bundles, and Trace Algebras}

\dedicatory{To the memory of Bill Arveson}

\author{Erin Griesenauer}

\address{Department of Mathematics \\
University of Iowa\\
 Iowa City, IA 52242}

\email{erin-griesenauer@uiowa.edu}

\author{Paul S. Muhly}

\address{Department of Mathematics \\
University of Iowa\\
 Iowa City, IA 52242}

\email{paul-muhly@uiowa.edu }

\author{Baruch Solel}

\address{Department of Mathematics\\
 Technion\\
32000 Haifa, Israel}

\email{mabaruch@techunix.technion.ac.il}

\maketitle
\begin{abstract}
We describe how noncommutative function algebras built from noncommutative
functions in the sense of \cite{K-VV2014} may be studied as subalgebras
of homogeneous $C^{*}$-algebras.
\end{abstract}

\section{Introduction\label{sec:Introduction}}

This note grew out of efforts to apply Arveson's boundary theory \cite{Arv1969a,Arv2008a,Arv2011,Arv2010b}
to operator algebras that arise naturally in free analysis. They are
built from the representation theory of free algebras, but our point
of view was inspired to a great extent by the recent book and perspective
of D. Kaliuzhnyi-Verbovetskyi and V. Vinnikov \cite{K-VV2014}. In
a sense, our purpose is to present ``a proof of concept''. The problem
which drew us to the topics discussed here remains unsolved. We will
discuss it in the final section, Section \ref{sec:Concluding-Remarks}.
Our efforts to solve this problem led us to methods from algebraic
geometry, geometric invariant theory and polynomial identity algebras
- subjects largely unfamiliar to us. Nevertheless, we hope to show
that these subjects carry useful information for free analysis and
its associated operator algebras. We have not striven for maximal
generality in the theorems and proofs presented in this paper. Rather,
we have tried to present a story whose purpose is to stimulate interest
among the operator algebra community in the algebras described here
and to stimulate future research. Consequently, the Introduction is
the bulk of the paper. It carries most of the narrative and the statements
of the main theorems. Most proofs and details are relegated to subsequent
(shorter) sections. 

The fundamental feature of the functions that we want to exploit is
that they are (\emph{holomorphic}) \emph{matrix concomitants}. Various
algebras they generate will be identified as subalgebras of homogeneous
$C^{*}$-algebras. To describe the functions and algebras, we need
to develop notation and provide background information. Throughout
this note $G$ will denote the projective linear group, $PGL(n,\mathbb{C})$,
which will be viewed as the group of \emph{automorphisms} of the full
algebra of complex $n\times n$ matrices, $M_{n}(\mathbb{C})$. The
subgroup of $G$ that preserves the usual $*$-structure on $M_{n}(\mathbb{C})$
is the projective unitary group, $PU(n,\mathbb{C})$. It will be denoted
by $K$. We frequently identify $G$ with $GL(n,\mathbb{C})$ and
write $s^{-1}as$, $a\in M_{n}(\mathbb{C})$, $s\in G$, for what
should be written as $a\cdot s$ or $s^{-1}\cdot a$. This should
cause no confusion since when $GL(n,\mathbb{C})$ appears in this
note, it always acts through conjugation of matrices. We study actions
of $G$ on $d$-tuples of $n\times n$ matrices, $M_{n}(\mathbb{C})^{d}$,
via the ``diagonal'' action. That is, we write elements of $\MnCd$
as $\mathfrak{z}=(Z_{1},Z_{2},\cdots,Z_{d})$, with $Z_{i}\in\MnC$,
and we write $\mathfrak{z}\cdot s=s^{-1}\mathfrak{z}s$ for $(s^{-1}Z_{1}s,s^{-1}Z_{2}s,\cdots,s^{-1}Z_{d}s)$,
$s\in G$. We are interested in domains $\mathcal{D}\subseteq\MnCd$
that are invariant under this action of $G$. A function $f$ defined
on such a domain $\mathcal{D}$ and mapping to $\MnC$ is called a
\emph{matrix concomitant} if $f$ satisfies the equation 
\begin{equation}
f(s^{-1}\mathfrak{z}s)=s^{-1}f(\mathfrak{z})s,\label{eq:HolMatCon}
\end{equation}
for all $s\in G$ and all $\mathfrak{z}\in\mathcal{D}$. The collection
of all holomorphic matrix concomitants defined on a domain $\mathcal{D}$
will be denoted $Hol(\mathcal{D},M_{n}(\mathbb{C}))^{G}$.  These
are the principal objects of study in this note. Unless explicitly
stated otherwise $d$ and $n$ will be assumed to be at least $2$
when discussing $d$-tuples of $n\times n$ matrices. 

Examples of holomorphic matrix concomitants are easy to come by. For
$i=1,2,\cdots,d,$ we let $\mathcal{Z}_{i}$ denote the function on
$M_{n}(\mathbb{C})^{d}$ defined by 
\[
\mathcal{Z}_{i}(\mathfrak{z}):=Z_{i},\qquad\mathfrak{z}=(Z_{1},Z_{2},\cdots,Z_{d}).
\]
That is, the $\mathcal{Z}_{i}$ are just the \emph{matrix} coordinate
functions defined on $M_{n}(\mathbb{C})^{d}$. Clearly, each $\mathcal{Z}_{i}$
is a holomorphic matrix concomitant. Since matrix concomitants form
an algebra under pointwise sums and products, the algebra generated
by the $\mathcal{Z}_{i}$ consists of holomorphic matrix concomitants.
This algebra is denoted $\mathbb{G}_{0}(d,n)$ and is called the \emph{algebra
of $d$ generic $n\times n$ matrices}. Evidently, it is the image
of the free algebra on $d$ variables, $\mathbb{C}\langle X_{1},X_{2},\cdots,X_{d}\rangle$
under the map that takes $X_{i}$ to $\mathcal{Z}_{i}$, $i=1,2,\cdots,d$.
Another important algebra of holomorphic matrix concomitants is built
from the algebra polynomial matrix \emph{invariants,} $\mathbb{I}_{0}(d,n)$,
which is the set of all polynomial functions $p:M_{n}(\mathbb{C})^{d}\to\mathbb{C}$
such that $p(s^{-1}\mathfrak{z}s)=p(\mathfrak{z})$, $s\in G$, $\mathfrak{z}\in M_{n}(\mathbb{C})^{d}$.
We identify $p\in\mathbb{I}_{0}(d,n)$ with the matrix-valued function
$\mathfrak{z}\to p(\mathfrak{z})I_{n}$, obtaining a polynomial matrix
concomitant. The algebra generated by $\mathbb{G}_{0}(d,n)$ and $\mathbb{I}_{0}(d,n)$
is denoted $\mathbb{S}_{0}(d,n)$ and is called the \emph{trace algebra}
of the generic matrices. In \cite[Theorem 2.1]{Procesi1976}, Procesi
proved that $\mathbb{S}_{0}(d,n)$ is precisely the set of all \emph{polynomial}
matrix concomitants. That is, $\mathbb{S}_{0}(d,n)$ consists of all
the matrix concomitants whose entries are polynomial functions of
$dn^{2}$ variables, organized as $d$-tuples of $n\times n$ matrices.
\begin{lem}
$Hol(M_{n}(\mathbb{C})^{d},M_{n}(\mathbb{C}))^{G}$ is the closure
of $\mathbb{S}_{0}(d,n)$ in topology of uniform convergence on compact
subsets of $M_{n}(\mathbb{C})^{d}$.\end{lem}
\begin{proof}
This is an easy application of Weyl's unitarian trick, which is often
regarded as the assertion that the maximal compact subgroup of a reductive
algebraic group is Zariski dense in the algebraic group \cite[Page 224 ff]{Procesi_2007}.
In our situation, it means that any polynomial function on $M_{n}(\mathbb{C})^{d}$
that is invariant under the action of $K$ is automatically invariant
under the action of $G$. Given $f\in Hol(M_{n}(\mathbb{C})^{d},M_{n}(\mathbb{C}))^{G}$
choose a sequence $\{p_{l}\}_{l\geq1}$ of $n\times n$ matrices,
whose entries are polynomial functions on $M_{n}(\mathbb{C})^{d}$,
that converges to $f$ uniformly on compact subsets of $M_{n}(\mathbb{C})^{d}$,
and define
\[
\widetilde{p}_{l}(\mathfrak{z})=\int_{K}kp_{l}(k^{-1}\mathfrak{z}k)k^{-1}\,dk,
\]
where ``$dk$'' denotes Haar measure on $K$. Then easy estimates
show that the $\widetilde{p}_{l}$ converge to $f$ uniformly on compact
subsets of $M_{n}(\mathbb{C})^{d}$. The $\widetilde{p}_{l}$ satisfy
the equation $\widetilde{p}_{l}(k^{-1}\mathfrak{z}k)=k^{-1}\widetilde{p}(\mathfrak{z})k$
for all $k\in K$. Since the $\widetilde{p}_{l}$ are all polynomials,
they are matrix concomitants by Weyl's unitarian trick, and Procesi's
theorem (loc. cit.) completes the proof.
\end{proof}
Procesi also proved in \cite[Theorem 3.4a]{Procesi1976} that $\mathbb{I}_{0}(d,n)$
is generated by the traces $tr(Z_{i_{1}}Z_{i_{2}}\cdots,Z_{i_{s}})$,
where $s\leq2^{n}-1$. Thus $\mathbb{I}_{0}(d,n)$ is finitely generated.
We may therefore consider the spectrum of $\mathbb{I}_{0}(d,n)$,
$Q(d,n)$, as an abstract affine algebraic variety defined over $\mathbb{C}$.
The inclusion of $\mathbb{I}_{0}(d,n)$ in the polynomial functions
mapping $M_{n}(\mathbb{C})^{d}$ to $\mathbb{C}$ induces, by way
of duality, a (regular) map $\pi_{0}$ from $M_{n}(\mathbb{C})^{d}$
onto $Q(d,n)$. 

If $\mathcal{V}(d,n)$ denotes the set of all $\mathfrak{z}=(Z_{1},Z_{2},\cdots,Z_{d})\in M_{n}(\mathbb{C})^{d}$
such that $Z_{1},Z_{2},\cdots,Z_{d}$ generate $M_{n}(\mathbb{C})$
as an algebra over $\mathbb{C}$, then $\mathcal{V}(d,n)$ is a $G$-invariant,
Zariski-open subset of $M_{n}(\mathbb{C})^{d}$, which we call the
set of \emph{irreducible points} of $M_{n}(\mathbb{C})^{d}$. Another
fundamental theorem of Procesi \cite[Theorem 5.10]{Procesi1974} asserts
that the image of $\mathcal{V}(d,n)$ under $\pi_{0}$, which we denote
by $Q_{0}(d,n)$, is an open subset of the smooth points of $Q(d,n)$
and that $(\mathcal{V}(d,n),\pi_{0},Q_{0}(d,n))$ has the structure
of a holomorphic principal $G$-bundle, denoted here by $\mathfrak{V}(d,n)$. 

We write $\mathfrak{M}(d,n)$ for the associated fibre bundle with
fibre $M_{n}(\mathbb{C})$, i.e., the bundle space of $\mathfrak{M}(d,n)$
is $\mathcal{V}(d,n)\times_{G}M_{n}(\mathbb{C})$, where $G$ acts
on $\mathcal{V}(d,n)\times M_{n}(\mathbb{C})$ via the formula $(\mathfrak{z},A)\cdot s=(s^{-1}\mathfrak{z}s,s^{-1}As)=(\mathfrak{z}\cdot s,s^{-1}\cdot A)$.
The projection $\pi:\mathcal{V}(d,n)\times_{G}M_{n}(\mathbb{C})\to Q_{0}(d,n)$
is given by formula $\pi([\mathfrak{z},A])=[\mathfrak{z}]$, in which
we adopt the convention that when $G$ acts on a set, say, $X$, then
the orbit of a point $x\in X$ is written $[x]$, i.e., $[x]:=\{x\cdot g\mid g\in G\}$.
Thus, in particular, $\pi([\mathfrak{z},A])=\pi_{0}(\mathfrak{z})$.

Our first result identifies the holomorphic cross sections of $\mathfrak{M}_{n}(d,n)$,
$\Gamma_{h}(Q_{0}(d,n),\mathfrak{M}(d,n))$, with the holomorphic
matrix concomitants on $\mathcal{V}(d,n)$. While the proof will be
presented in Section \ref{sec:Proof_of_thm_Cross section}, it will
be helpful to reflect here on the connection between cross sections
and concomitants. Everything boils down to parsing this equation:
\begin{equation}
\sigma([\mathfrak{z}])=[\mathfrak{z},\phi(\mathfrak{z})],\label{eq:Fund_equat}
\end{equation}
$\mathfrak{z}\in\mathcal{V}(d,n)$, where $\sigma$ is a cross section
of $\mathfrak{M}(d,n)$ and $\phi$ is a matrix concomitant. The key
for this is to note that if we are given $\mathfrak{u}\in Q_{0}(d,n)$
and $\mathfrak{a}\in\mathfrak{M}(d,n)$ such that $\pi(\mathfrak{a})=\mathfrak{u}$,
then once $\mathfrak{z}\in\mathcal{V}(d,n)$ is chosen so that $\pi_{0}(\mathfrak{z})=\mbox{\ensuremath{\mathfrak{u}}}$,
there is one and only one $A\in M_{n}(\mathbb{C})$ such that $\mathfrak{a}=[\mathfrak{z},A]$.
Now let's read \eqref{eq:Fund_equat} from left to right and suppose
$\sigma$ is a cross section of $\mathfrak{M}(d,n)$. If $\mathfrak{u}\in Q_{0}(d,n)$,
then for $\mathfrak{z}\in\pi_{0}^{-1}(\mathfrak{u}),$ there is one
and only one matrix $\phi(\mathfrak{z})\in M_{n}(\mathbb{C})$ such
that $[\mathfrak{z},\phi(\mathfrak{z})]=\sigma(\mathfrak{u})$. This
defines $\phi$ on $\pi_{0}^{-1}(\mathfrak{u})$ for each $\mathfrak{u}\in Q_{0}(d,n)$,
and so the $M_{n}(\mathbb{C})$-valued function, $\phi$, is well
defined on all of $\mathcal{V}(d,n)$. On the other hand, $\pi_{0}(\mathfrak{z}\cdot s)=\mathfrak{u}$
for any $s\in G$. So $\pi([\mathfrak{z}\cdot s,\phi(\mathfrak{z}\cdot s)])=\mathfrak{u}$,
too. But by definition of the action of $G$ on $\mathcal{V}(d,n)\times M_{n}(\mathbb{C})$,
$[\mathfrak{z}\cdot s,\phi(\mathfrak{z}\cdot s)]=[\mathfrak{z},s\cdot\phi(\mathfrak{z}\cdot s)]$,
which shows that $s\cdot\phi(\mathfrak{z}\cdot s)=\phi(\mathfrak{z})$,
i.e., $\phi(\mathfrak{z}\cdot s)=s^{-1}\phi(\mathfrak{z})s$. Reading
\eqref{eq:Fund_equat} from right to left, suppose $\phi$ is a matrix
concomitant on $\mathcal{V}(d,n)$. Then $[\mathfrak{z},\phi(\mathfrak{z})]$
is an element in $\mathfrak{M}(d,n)$ such that $\pi([\mathfrak{z},\phi(\mathfrak{z})])=\pi_{0}(\mathfrak{z})=[\mathfrak{z}]$.
But for each $s\in G$, $\pi([\mathfrak{z}\cdot s,\phi(\mathfrak{z}\cdot s)])=\pi_{0}(\mathfrak{z}\cdot s)=[\mathfrak{z}]$,
too, and $[\mathfrak{z}\cdot s,\phi(\mathfrak{z}\cdot s)]=[\mathfrak{z},s\cdot\phi(\mathfrak{z}\cdot s)]=[\mathfrak{z},\phi(\mathfrak{z})]$
because $\phi$ is a concomitant. Therefore, if we set $\sigma([\mathfrak{z}])=[\mathfrak{z},\phi(\mathfrak{z})]$,
then $\sigma$ is well defined.

Henceforth, then, given a matrix concomitant $\phi$, we shall write
$\sigma_{\phi}$ for the cross section of $\mathfrak{M}(d,n)$ determined
by $\phi$ via \eqref{eq:Fund_equat} and conversely, given a cross
section $\sigma$ of $\mathfrak{M}(d,n)$, we shall write $\phi_{\sigma}$
for the matrix concomitant defined through \eqref{eq:Fund_equat}. 
\begin{thm}
\label{thm: Cross sectional Algebra} For $d\geq2$ and $n\geq2$,
the correspondence $\phi\to\sigma_{\phi}$ defines an algebra isomorphism
$\Psi$ from $Hol(\mathcal{V}(d,n),M_{n}(\mathbb{C}))^{G}$ onto $\Gamma_{h}(Q_{0}(d,n),\mathfrak{M}(d,n))$,
with inverse given by $\sigma\to\phi_{\sigma}$. If, in addition,
$d$ or $n$ is greater than $2$, then every concomitant in $Hol(\mathcal{V}(d,n),M_{n}(\mathbb{C}))^{G}$
admits a unique extension to a concomitant in $Hol(M_{n}(\mathbb{C})^{d},M_{n}(\mathbb{C}))^{G}$.
The domain $\mathcal{V}(2,2)$, on the other hand, is a domain of
holomorphy and there are concomitants in $Hol(\mathcal{V}(2,2),M_{2}(\mathbb{C}))^{G}$
that do not extend to $M_{2}(\mathbb{C})^{2}$.
\end{thm}
Theorem \ref{thm: Cross sectional Algebra} gives a faithful representation
of $Hol(\mathcal{V}(d,n),M_{n}(\mathbb{C}))^{G}$ as a space of functions
on the space of similarity classes of its irreducible matrix representations.
It has the following immediate corollary.
\begin{cor}
\label{cor:Integral_domain} The bundle $\mathfrak{M}(d,n)$ is not
trivial when $(d,n)\neq(2,2)$.\end{cor}
\begin{proof}
By \cite[Proposition 4.4]{Lum97}, $Hol(M_{n}(\mathbb{C})^{d},M_{n}(\mathbb{C}))^{G}$
has no zero divisors. Since $Hol(M_{n}(\mathbb{C})^{d},M_{n}(\mathbb{C}))^{G}$
and $\Gamma_{h}(Q_{0}(d,n),\mathfrak{M}(d,n))$ are isomorphic when
$(d,n)\neq(2,2)$, neither does $\Gamma_{h}(Q_{0}(d,n),\mathfrak{M}(d,n))$
in this case. However, if $\mathfrak{M}(d,n)$ were trivial, $\Gamma_{h}(Q_{0}(d,n),\mathfrak{M}(d,n))$
would be isomorphic to the $n\times n$ matrices over the space of
holomorphic functions on $Q_{0}(d,n)$, which has plenty of zero divisors.
\end{proof}
Presumably, $\mathfrak{M}(2,2)$ is nontrivial, too, but we do not
know a proof.

Our focus then turns to domains $\mathcal{D}$ such that $\overline{\mathcal{D}}$
is a compact subset of $Q_{0}(d,n)$. Since $Q(d,n)$ is the spectrum
of $\mathbb{I}_{0}(d,n)$, the image of $\mathbb{I}_{0}(d,n)$ under
$\Psi$ coincides with the algebra of regular $\mathbb{C}$-valued
functions on $Q(d,n$). That is, if $\mathfrak{w}\in Q(d,n)$ and
if $\mathfrak{z}\in M_{n}(\mathbb{C})^{d}$ is such that $\pi_{0}(\mathfrak{z})=\mathfrak{w},$
then for $f\in\mathbb{I}_{0}(d,n)$, we get $\Psi(f)(\mathfrak{w})=f(\mathfrak{z})$,
identified with the cross section of $\mathfrak{M}(d,n)$ that $f$
determines. That is, $\Psi(f)([\mathfrak{z}])=[\mathfrak{z},f(\mathfrak{z})]$.
We let $\mathbb{I}(\mathcal{D};d,n)$ denote the closure of $\{\Psi(f)\mid f\in\mathbb{I}_{0}(d,n)\}$
in the space of continuous $\mathbb{C}$-valued functions on $\overline{\mathcal{D}}$,
$C(\overline{\mathcal{D}})$. Since $\mathbb{I}_{0}(d,n)$ contains
the constant functions and separates the points of $Q(d,n)$, $\mathbb{I}(\mathcal{D};d,n)$
is a function algebra on $\overline{\mathcal{D}}$, consisting of
functions that are continuous on $\overline{\mathcal{D}}$ and holomorphic
on $\mathcal{D}$. Although $\overline{\mathcal{D}}$ need not be
the maximal ideal space of $\mathbb{I}(\mathcal{D};d,n)$, $\overline{\mathcal{D}}$
contains the Shilov boundary of the maximal ideal space, which we
denote by $\partial\mathcal{D}$. (This is the case simply because
$\mathbb{I}(\mathcal{D};d,n)$ is a function algebra on $\overline{\mathcal{D}}$.)
The extreme boundary, or Choquet boundary of $\mathcal{D}$, will
be denoted $\partial_{e}\mathcal{D}$. It is a dense subset of $\partial\mathcal{D}$
that consists of all points in $\overline{\mathcal{D}}$ that have
unique representing measures for $\mathbb{I}(\mathcal{D};d,n)$ supported
in $\overline{\mathcal{D}}.$ 

We are interested both in the holomorphic cross sections of $\mathfrak{M}(d,n)$
and in its continuous cross sections, $\Gamma_{c}(Q_{0}(d,n),\mathfrak{M}(d,n))$.
The problem we face is that there is no \emph{evident} natural involution
on $\mathfrak{M}(d,n)$ with respect to which $\Gamma_{c}(X,\mathfrak{M}(d,n))$
is a $C^{*}$-algebra for every compact subset $X\subseteq Q_{0}(d,n)$.
This is because $\mathcal{V}(d,n)$ is a principal $G$-bundle and
so in a coordinate representation of $\mathcal{V}(d,n)$ the transition
functions need not take their values in $K$. In fact, $\Gamma_{c}(X,\mathfrak{M}(d,n))$
does not carry a \emph{canonical} Banach algebra structure. Nevertheless,
there are many \emph{ad hoc} Banach algebra structures on $\Gamma_{c}(X,\mathfrak{M}(d,n))$,
which may be constructed as follows. Take a locally finite open cover
$\mathcal{U}$ of $Q_{0}(d,n)$ with an associated set of transition
functions $\{g_{UV}\}_{U,V\in\mathcal{U}}$ that define $\mathcal{V}(d,n)$
as a principal bundle. Then take isomorphisms $F_{U}:\mathfrak{M}(d,n)\vert_{U}\to U\times M_{n}(\mathbb{C})$
that allow one to identify continuous cross sections of $\mathfrak{M}(d,n)$
over $U$ with continuous $M_{n}(\mathbb{C})$-valued functions $f_{U}$
on $U$ that satisfy $f_{U}(\mathfrak{u})=g_{UV}(\mathfrak{u})\circ f_{V}(\mathfrak{u})$
on $U\cap V$. For a given compact subset $X\subseteq Q_{0}(d,n)$
one can then define a Banach algebra norm on $\Gamma_{c}(X,\mathfrak{M}(d,n))$
by setting
\begin{equation}
\Vert\sigma\Vert_{\mathcal{U}}:=\sup_{x\in X}\sup_{x\in U}\Vert F_{U}(\sigma)(x)\Vert,\qquad\sigma\in\Gamma_{c}(X,\mathfrak{M}(d,n)).\label{eq:Ad_hoc_BA-structure}
\end{equation}
Here the norm $\Vert F_{U}(\sigma)(x)\Vert$ refers to the Hilbert
space operator norm one obtains by viewing $M_{n}(\mathbb{C})$ as
operators on $\mathbb{C}^{n}$ in the usual way. Different systems
of data $(\mathcal{U},\{g_{UV}\}_{U,V\in\mathcal{U}},\{F_{U}\}_{U\in\mathcal{U}})$
give different norms, but the norms are all equivalent, i.e., the
Banach algebras constructed are mutually isomorphic, and they all
yield the compact-open topology on $\Gamma_{c}(X,\mathfrak{M}(d,n))$
for any compact set $X\subseteq Q_{0}(d,n)$. 

It may come as a pleasant surprise, therefore, to learn that there
\emph{is} a way to put a $C^{*}$-algebra strcture on $\Gamma_{c}(X,\mathfrak{M}(d,n))$
for each compact set $X\subseteq Q_{0}(d,n)$. In fact, any two $C^{*}$-algebra
structures on $\Gamma_{c}(X,\mathfrak{M}(d,n))$ are $*$-isomorphic.
We must emphasize the difference between `isomorphic' and `equal'
here because the isomorphisms involved almost always map some holomorphic
sections to non-holomorphic sections. Each $C^{*}$-structure on $\Gamma_{c}(X,\mathfrak{M}(d,n))$
is obtained from a \emph{reduction} $\mathfrak{P}$ of $\mathcal{V}(d,n)$
to a principal $K$-bundle over $X$\footnote{We follow Steenrod \cite{Steenrod1951} in the use of the term ``reduction''.
Husemoller uses the term ``restriction''.}. For our purposes, this means that $\mathfrak{P}$ is a principal
$K$ bundle obtained from a $K$-invariant compact subset $\mathcal{P}$
of $\mathcal{V}(d,n)$ that $\pi$ maps onto $X$. That is, $\pi$
identifies $X$ with $\mathcal{P}/K$. From a coordinate point of
view, the transition functions defining $\mathfrak{P}$ take their
values in $K$ and so the associated $M_{n}(\mathbb{C})$-fibre bundle,
which we denote by $\mathfrak{M}^{*}(\mathfrak{P};d,n)$, has a natural,
fibre-wise-defined involution. The bundles $\mathfrak{M}(d,n)$ and
$\mathfrak{M}^{*}(\mathfrak{P};d,n)$ are isomorphic as topological
bundles \cite[Theorem 6.3.1]{Husemoller1994}. Therefore for any compact
subset $X$ of $Q_{0}(d,n)$, $\Gamma_{c}(X,\mathfrak{M}^{*}(\mathfrak{P};d,n))$
and $\Gamma_{c}(X,\mathfrak{M}(d,n))$ are isomorphic Banach algebras,
where $\Gamma_{c}(X,\mathfrak{M}(d,n))$ is given any of the norms
$\Vert\cdot\Vert_{\mathcal{U}}$ defined in \eqref{eq:Ad_hoc_BA-structure}
using a choice of the data $(\mathcal{U},\{g_{UV}\}_{U,V\in\mathcal{U}},\{F_{U}\}_{U\in\mathcal{U}})$. 

In the norm on $\GcDM$, elements in $\GhDM$ achieve their maximums
on $\partial D$. However, it is easy to construct examples of reductions
$\mathfrak{P}$ of $\mathcal{V}(d,n)$ such that the image of an element
from $\GhDM$ in $\Gamma_{c}(\overline{\mathcal{D}},\mathfrak{M}^{*}(\mathfrak{P};d,n))$
need not take its maximum norm on $\partial\mathcal{D}$. For this
reason, we adjust our focus and concentrate \emph{directly} on $\Gamma_{c}(\partial\mathcal{D},\mathfrak{M}^{*}(\mathfrak{P};d,n))$.
\begin{defn}
\label{def:NC-fcn-algebra} The closure of $\Psi(\mathbb{S}_{0}(d,n))$
in $\Gamma_{c}(\partial\mathcal{D},\mathfrak{M}^{*}(\mathfrak{P};d,n))$
will be denoted $\mathbb{S}(\mathcal{D},\mathfrak{P};d,n)$ and will
be called the \emph{tracial function algebra} of $\mathcal{D}$ determined
by $\mathfrak{P}$ and $\mathbb{S}_{0}(d,n)$. 
\end{defn}
Observe that when $n=1$, $G=K$ is the trivial group; $\mathcal{V}(d,n)$,
$\mathcal{P}$, and $Q_{0}(d,n)$ become identified with $\mathbb{C}^{d}\backslash\{0\}$;
$\mathfrak{M}(d,n)=\mathfrak{M}^{*}(\mathfrak{P};d,n)$ is the trivial
line bundle on $\mathbb{C}^{d}$; and the algebras $\mathbb{I}(\mathcal{D};d,n)$
and $\mathbb{S}(\mathcal{D},\mathfrak{P};d,n)$ are identified with
$\mbox{\ensuremath{\mathcal{P}}(\ensuremath{\overline{\mathcal{D}}})}$,
the $\sup$-norm closure of the polynomial functions on $\mathbb{C}^{d}$
in the continuous functions on $\overline{\mathcal{D}}$. Of course,
$\mathcal{P}(\overline{\mathcal{D}})$ is a much studied algebra in
complex analysis (see, e.g. \cite{Stout_2007}), but there does not
seem to be a universally accepted term for it. Our current thinking
is that $\mathbb{S}(\mathcal{D},\mathfrak{P};d,n)$ is the natural
generalization of $\mathcal{P}(\overline{\mathcal{D}})$.

We note that the center of $\mathbb{S}(\mathcal{D},\mathfrak{P};d,n)$
may be identified in a natural fashion with $\mathbb{I}(\mathcal{D};d,n)$,
no matter what reduction is chosen. We shall give a proof of this
fact in Section \ref{sec:Function-Theory-in-Bundles}. The reason
the assertion is true is that elements of $\mathbb{I}(\mathcal{D};d,n)$
are identified with sections whose values are scalar multiples of
the identity and these are unaffected by the transition functions
that describe the bundles. The fact that the center of $\mathbb{S}(\mathcal{D},\mathfrak{P};d,n)$
is $\mathbb{I}(\mathcal{D};d,n)$ shows in particular that $\mathbb{S}(\mathcal{D},\mathfrak{P};d,n)$
is a \emph{proper} subalgebra of $\Gamma_{c}(\partial\mathcal{D},\mathfrak{M}^{*}(\mathfrak{P};d,n))$.
This is not evident, \emph{a priori}. 

The $C^{*}$-algebra $\Gamma_{c}(\partial\mathcal{D},\mathfrak{M}^{*}(\mathfrak{P};d,n))$
is an $n$-homogeneous $C^{*}$-algebra \cite[Theorem 8]{Tomiyama1961}
and each irreducible representation of it is given, essentially, by
evaluation at a unique point of $\partial\mathcal{D}$. In more detail,
note that for $\mathfrak{u}\in Q_{0}(d,n)$, $\pi^{-1}(\mathfrak{u})=\{[\mathfrak{z},A]\in\mathcal{P}\times_{K}M_{n}(\mathbb{C})\mid\pi_{0}(\mathfrak{z})=\mathfrak{u},\,A\in M_{n}(\mathbb{C})\}$.
So, once $\mathfrak{z}$ is chosen so that $\pi_{0}(\mathfrak{z})=\mathfrak{u}$
the map $A\to[\mathfrak{\mathfrak{z}},A]$ is a unital $*$-homomorphism
$\rho$ of $M_{n}(\mathbb{C})$ into $\pi^{-1}(\mathfrak{u})$. Since
$M_{n}(\mathbb{C})$ is simple, the map is injective. It is surjective
because if $[\mathfrak{w},B]$ lies in $\pi^{-1}(\mathfrak{u})$,
then there is a unique $s\in K$ such that $\mathfrak{w}=\mathfrak{z}\cdot s$
and we may write: $[\mathfrak{w},B]=[\mathfrak{z}\cdot s,B]=[\mathfrak{z},s\cdot B]$,
which is in the image of $\rho$. Thus, if for each $\mathfrak{u}\in\partial\mathcal{D}$,
we write $ev_{\mathfrak{u}}$ for the $*$-homomorphism from $\Gamma_{c}(\partial\mathcal{D},\mathfrak{M}^{*}(\mathfrak{P};d,n))$
into $\pi^{-1}(\mathfrak{u})$ defined by evaluating a section in
$\Gamma_{c}(\partial\mathcal{D},\mathfrak{M}^{*}(\mathfrak{P};d,n))$
at $\mathfrak{u}$, then $\rho^{-1}\circ ev_{\mathfrak{u}}$ is an
irreducible representation of $\Gamma_{c}(\partial\mathcal{D},\mathfrak{M}^{*}(\mathfrak{P};d,n))$
and every irreducible representation of $\Gamma_{c}(\partial\mathcal{D},\mathfrak{M}^{*}(\mathfrak{P};d,n))$
is unitarily equivalent to $\rho^{-1}\circ ev_{\mathfrak{u}}$ for
a unique $\mathfrak{u}\in\partial\mathcal{D}$ by \cite[Corollary 10.4.4]{Dixmier}. 

The two principal theorems of this note are Theorems \ref{thm:Boundary theorem}
and \ref{thm:Azumaya}, below. For the first, and its corollary, Corollary
\ref{cor:Shilov_boundary_ideal}, we need to recall Arveson's definition
of a boundary representation, and related ideas.
\begin{defn}
\label{def:_boundary_rep} \cite[Definition 2.1.1]{Arv1969a} If $B$
is a unital $C^{*}$-algebra and if $A$ is a norm-closed subalgebra
of $B$ that contains the unit of $B$ and generates $B$ as a $C^{*}$-algebra,
then an irreducible representation $\pi:B\to B(H_{\pi})$ is a \emph{boundary
representation} for $A$ in case $\pi$ is the only unital completely
positive map $\omega:B\to B(H_{\pi})$ such that $\pi\vert_{A}=\omega\vert_{A}$. \end{defn}
\begin{thm}
\label{thm:Boundary theorem} If $\mathfrak{u}\in\partial_{e}\mathcal{D}$,
then $ev_{\mathfrak{u}}$ is a boundary representation of $\Gamma_{c}(\partial\mathcal{D},\mathfrak{M}^{*}(\mathfrak{P};d,n))$
for $\mathbb{S}(\mathcal{D},\mathfrak{P};d,n)$. 
\end{thm}
In the setting of Definition \ref{def:_boundary_rep}, an ideal $\mathfrak{I}$
in $B$ is called a \emph{boundary ideal} in case the restriction
to $A$ of the quotient map $q:B\to B/\mathfrak{I}$ is completely
isometric. The intersection of the kernels of the boundary representations
of $B$ for $A$ is the largest boundary ideal, which is called the
\emph{Shilov boundary ideal} of $B$ for $A$\footnote{When \cite{Arv1969a} was written, it was not known if boundary representations
always exist and the Shilov boundary ideal was defined differently;
the existence of the Shilov boundary ideal was problematic. Today,
thanks to \cite{Arv2008a} and \cite{Davidson_and_Kennedy_2013},
it is known that in every setting there are sufficiently many boundary
representations to determine the Shilov boundary ideal.}. The quotient of $B$ by the Shilov boundary ideal is unique up to
$C^{*}$-isomorphism in a very strong sense \cite[Theorem 2.2.6]{Arv1969a}.
The quotient is called the\emph{ $C^{*}$-envelope of }$A$. 
\begin{cor}
\label{cor:Shilov_boundary_ideal}For each reduction $\mathfrak{P}$
of $\mathcal{V}(d,n)$ and for each domain $\mathcal{D}$ with $\overline{\mathcal{D}}$
contained in $Q_{0}(d,n)$, the Shilov boundary ideal of $\Gamma_{c}(\partial\mathcal{D},\mathfrak{M}^{*}(\mathfrak{P};d,n))$
for $\mathbb{S}(\mathcal{D},\mathfrak{P};d,n)$ vanishes, so $\Gamma_{c}(\partial\mathcal{D},\mathfrak{M}^{*}(\mathfrak{P};d,n))$
is the $C^{*}$-envelope of $\mathbb{S}(\mathcal{D},\mathfrak{P};d,n)$.
\end{cor}
{}

The pair $(Q(d,n),\mathbb{I}_{0}(d,n))$ is an example of what Rickart
calls a natural function algebra \cite{Rickart_1979}, where $Q(d,n)$
is considered with its analytic topology. If $X\subseteq Q(d,n)$
is a compact subset, then the \emph{$\mathbb{I}_{0}(d,n)$-convex
hull of }$X$, $\widehat{X}$, is defined to be $\{\mathfrak{z}\in Q(d,n)\mid\vert f(\mathfrak{z})\vert\leq\Vert f\Vert_{X},\,f\in\mathbb{I}_{0}(d,n)\}$,
where $\Vert f\Vert_{X}:=\sup_{\mathfrak{z}\in X}\vert f(\mathfrak{z})\vert$.
If $X=\widehat{X}$, then $X$ is called\emph{ $\mathbb{I}_{0}(d,n)$-convex}.
The maximal ideal space of the closure of $\mathbb{I}_{0}(d,n)$ in
$C(X)$ is $\widehat{X}$.

We note in passing that when $d=n=2$, $\mathbb{I}_{0}(d,n)$ is isomorphic
to the polynomial algebra in five variables; so $Q(2,2)$ may be identified
with $\mathbb{C}^{5}$ (see, e.g., \cite[P. 14 ff]{LeBruyn2008}).
Thus, in this case, the $\mathbb{I}_{0}(2,2)$-convex hull of a compact
set $X$ coincides with its polynomially convex hull. In general,
however, $\mathbb{I}_{0}(d,n)$ is more complicated and still largely
mysterious. It is worth noting that when $d=n=2$, the identification
of $Q(2,2)$ with $\mathbb{C}^{5}$ is through the map
\[
(Z_{1},Z_{2})\to({\rm tr}(Z_{1}),{\rm tr}(Z_{2}),\det(Z_{1}),\det(Z_{2}),{\rm tr}(Z_{1}Z_{2})).
\]

So even in this setting, the interaction of the map with the norms
involved is unclear. The situation is further complicated by the fact
that generators of $\mathbb{I}_{0}(2,2)$ are not uniquely determined
and it is not at all clear which ones are best for, or even well adapted
to, analysis.
\begin{defn}
A unital algebra $\mathfrak{A}$ with center $\mathfrak{Z}$ is called
an \emph{Azumaya algebra} in case
\begin{enumerate}
\item As a right module over $\mathfrak{Z}$, $\mathfrak{A}$ is projective,
and
\item The map from $\mathfrak{A}\otimes_{\mathfrak{Z}}\mathfrak{A}^{op}$
to $End(\mathfrak{A}_{\mathfrak{Z}})$ defined by identifying $a\otimes b$
with the endomorphism
\[
a\otimes b(c):=acb,\qquad c\in\mathfrak{A},
\]
is an isomorphism.
\end{enumerate}
\end{defn}
This is one of many equivalent definitions. For further background
on such algebras, see \cite{DeMeyer-Ingraham_1971}. The importance
of these algebras for us is that they are algebraic versions of $n$-homogeneous
$C^{*}$-algebras by \cite[Theorem 8.3]{Artin1969}. Specifically,
Artin proved in his Theorem 8.3 (specialized to algebras over $\mathbb{C}$)
that if $A$ is a unital $\mathbb{C}$-algebra, then $A$ is an Azumaya
algebra of rank $n^{2}$ over its center if and only if $A$ satisfies
the identities of the $n\times n$ matrices and $A$ has no (unital)
representations in $M_{r}(\mathbb{C})$ for $r\lneq n$. (To say in
this setting that $A$ has rank $n^{2}$ over its center means that
for each maximal $2$-sided ideal $\mathfrak{m}$ of $A$, $A/\mathfrak{m}\simeq M_{n}(\mathbb{C})$.)
Equivalently, under the hypothesis that $A$ satisfies the identities
of the $n\times n$ matrices, the theorem asserts that $A$ is an
Azumaya algebra if and only if each (algebraically) irreducible representation
of $A$ is $n$-dimensional. Artin was inspired, in part, by Tomiyama
and Takesaki's representation of an $n$-homogeneous $C^{*}$-algebra
as the continuous cross sections of a matrix bundle in \cite{Tomiyama1961}.
Thus, in one sense, the following theorem may easily be anticipated,
given that the algebra in question is a subalgebra of an $n$ homogeneous
$C^{*}$-algebra. However, the proof may not seem immediate. Further,
the theorem has consequences that appear difficult to establish without
it, e.g., Corollary \ref{cor:ideals}.
\begin{thm}
\label{thm:Azumaya} If $\overline{\mathcal{D}}$ is $\mathbb{I}_{0}(d,n)$-convex,
then the algebra $\mathbb{S}(\mathcal{D},\mathfrak{P};d,n)$ is a
rank $n^{2}$ Azumaya algebra over $\mathbb{I}(\mathcal{D};d,n)$.\end{thm}
\begin{cor}
\label{cor:ideals}If $\overline{\mathcal{D}}$ is $\mathbb{I}_{0}(d,n)$-convex,
then there is a bijective correspondence between ideals $\mathfrak{a}$
of $\mathbb{I}(\mathcal{D};d,n)$ and ideals $\mathfrak{A}$ of $\mathbb{S}(\mathcal{D},\mathfrak{P};d,n)$
given by $\mathfrak{a}\to\mathfrak{a}\mathbb{S}(\mathcal{D},\mathfrak{P};d,n)$
and $\mathfrak{A}\to\mathfrak{A}\cap\mathbb{I}(\mathcal{D};d,n)$.\end{cor}
\begin{proof}
This is an application of Corollary II.3.7 of \cite{DeMeyer-Ingraham_1971},
which is valid for any Azumaya algebra.
\end{proof}

\section{The Concomitants and Cross Sections\label{sec:Proof_of_thm_Cross section}}

The map we call $\Psi$ in Theorem \ref{thm: Cross sectional Algebra}
is a special case of the bijection described in \cite[Theorem 4.8.1]{Husemoller1994}.
There, Husemoller deals with general fibre bundles associated to principal
bundles. However, when specialized to our setting it is clear that
$\Psi$ is a bijection that takes continuous concomitants to continuous
cross sections. It also clearly preserves the algebraic structures
involved. So to prove Theorem \ref{thm: Cross sectional Algebra},
it suffices to show that $\Psi$ maps holomorphic concomitants to
holomorphic cross sections and that $\Psi^{-1}$ maps holomorphic
cross sections to holomorphic concomitants. 

Since the property of being holomorphic is a local property, we may
restrict our attention to an open subset $U\subseteq Q_{0}(d,n)$
over which $\mathcal{V}(d,n)$ is trivial. We let $\mathcal{V}_{0}=\pi_{0}^{-1}(U)$,
so $\mathcal{V}_{0}$ is an open, $G$-invariant subset of $\mathcal{V}(d,n)$,
and we fix a biholomorphic bundle isomorphism $F:U\times G\to\mathcal{V}_{0}$.
Thus $F$ is $G$-equivariant and $\pi_{0}\circ F=\pi_{1}$, where
$\pi_{1}$ is the projection of $U\times G$ onto the first factor.
(This implies that $\mathfrak{u}\to F(\mathfrak{u},e)$ is a holomorphic
section of $\mathfrak{V}(d,n)\vert_{U}$, and conversely, each holomorphic
section $f$ of $\mathfrak{V}(d,n)\vert_{U}$ determines a biholomorphic
bundle isomorphism from $U\times G$ onto $\mathcal{V}_{0}$ via the
formula $F(\mathfrak{u},g)=f(\mathfrak{u})g$.) The isomorphism $F$,
in turn, induces a biholomorphic bundle isomorphism $\widehat{F}:U\times M_{n}(\mathbb{C})\to\mathcal{V}_{0}\times_{G}M_{n}(\mathbb{C})$
via the formula $\widehat{F}(\mathfrak{u},A)=[F(\mathfrak{u},e),A]$. 

Suppose that $\phi:\mathcal{V}(d,n)\to M_{n}(\mathbb{C})$ is a holomorphic
matrix concomitant. Then the restriction to $U$ of the section $\sigma_{\phi}$
defined above is given by the formula 
\[
\sigma_{\phi}(\mathfrak{u})=[\mathfrak{z},\phi(\mathfrak{z})],\qquad\mathfrak{u}\in U,
\]
where $\mathfrak{z}\in\mathcal{V}_{0}$ is any point such that $\pi_{0}(\mathfrak{z})=\mathfrak{u}$.
To show $\sigma_{\phi}$ is holomorphic on $U$, it suffices to show
that $\widehat{F}^{-1}\circ\sigma_{\phi}$ is holomorphic on $U$.
To get a formula for $\widehat{F}^{-1}\circ\sigma_{\phi}$, fix both
$\mathfrak{u}\in U$ and $\mathfrak{z}\in\mathcal{V}_{0}$ such that
$\pi_{0}(\mathfrak{z})=\mathfrak{u}.$ Then there is a unique $g\in G$
such that $F(\mathfrak{u},g)=\mathfrak{z}$. Since we also have $F(\mathfrak{u},g)=F(\mathfrak{u},e)g$,
we arrive at the following equation,
\begin{multline*}
\widehat{F}^{-1}\circ\sigma_{\phi}(\mathfrak{u})=\widehat{F}^{-1}([\mathfrak{z},\phi(\mathfrak{z})])=\widehat{F}^{-1}([F(\mathfrak{u},g),\phi(\mathfrak{z})])\\
=\widehat{F}^{-1}([F(\mathfrak{u},e),g\cdot\phi(\mathfrak{z})])=\widehat{F}^{-1}(\widehat{F}(\mathfrak{u},\phi(\mathfrak{z}\cdot g^{-1})))\\
=(\mathfrak{u},\phi(\mathfrak{z}\cdot g^{-1}))=(\mathfrak{u},\phi\circ F(\mathfrak{u},e)),
\end{multline*}
which shows that $\widehat{F}^{-1}\circ\sigma_{\phi}$ is holomorphic
on $U$, since $\mathfrak{u}\to(\mathfrak{u},\phi\circ F(\mathfrak{u},e))$
is certainly holomorphic.

If $\sigma$ is a holomorphic section of $\mathfrak{M}(d,n)$, then
to show that $\phi_{\sigma}$ is holomorphic, it suffices to show
that the restruction of $\phi_{\sigma}$ to $\mathcal{V}_{0}$ is
holomorphic; and for this, it suffices to show that $\phi_{\sigma}\circ F$
is holomorphic on $U\times G$. Since $\mathfrak{M}(d,n)$ is trivial
over $U$ and $\sigma\vert_{U}$ is a section of $\mathfrak{M}(d,n)\vert_{U}$
, $\widehat{F}^{-1}\circ\sigma\vert_{U}$ a section of the product
bundle $U\times M_{n}(\mathbb{C})$ over $U$. Consequently, there
is a function $f:U\to M_{n}(\mathbb{C})$ such that $\widehat{F}^{-1}\circ\sigma(\mathfrak{u})=(\mathfrak{u},f(\mathfrak{u}))$.
The assumption that $\sigma$ is holomorphic guarentees that $f$
is holomorphic, too. On the other hand, the matrix concomitant $\phi_{\sigma}$
determined by $\sigma$ satisfies \eqref{eq:Fund_equat}. Therefore,
$(\mathfrak{u},f(\mathfrak{u}))=\widehat{F}^{-1}\circ\sigma(\mathfrak{u})=\widehat{F}^{-1}([\mathfrak{z},\phi_{\sigma}(\mathfrak{z})])$
for any $\mathfrak{z}$ such that $\pi_{0}(\mathfrak{z})=\mathfrak{u}.$
So, 
\[
\widehat{F}(\mathfrak{u},f(\mathfrak{u}))=[\mathfrak{z},\phi_{\sigma}(\mathfrak{z})].
\]
However, by definition of $\widehat{F}$ in terms of $F$, we may
rewrite the left-hand side of this equation as
\begin{multline*}
\widehat{F}(\mathfrak{u},f(\mathfrak{u}))=[F(\mathfrak{u},e),f(\mathfrak{u})]\\
=[F(\mathfrak{u},e)\cdot g,g^{-1}\cdot f(\mathfrak{u})]=[F(\mathfrak{u},g),g^{-1}\cdot f(\mathfrak{u})].
\end{multline*}

If we write $\mathfrak{z}=F(\mathfrak{u},g)$, these two equations
yield
\[
[F(\mathfrak{u},g),g^{-1}\cdot f(\mathfrak{u})]=\widehat{F}(\mathfrak{u},f(\mathfrak{u}))=[F(\mathfrak{u},g),\phi_{\sigma}(F(\mathfrak{u},g))].
\]
Hence, there is an $h\in G$ such $F(\mathfrak{u},g)\cdot h=F(\mathfrak{u},g)$
and $h^{-1}\cdot g^{-1}\cdot f(\mathfrak{u})=\phi_{\sigma}(F(\mathfrak{u},g))$.
However, since $G$ acts freely on $\mathcal{V}(d,n)$, we conclude
that $h=e$, proving that
\[
\phi_{\sigma}\circ F(\mathfrak{u},g)=g^{-1}\cdot f(\mathfrak{u}).
\]
Since $f$ is holomorphic on $U$ and the action of $G$ on $M_{n}(\mathbb{C})$
is holomorphic, we see that $\phi_{\sigma}\circ F$ is holomorphic
on $U\times G$, as required. This completes the proof of the first
assertion in Theorem \ref{thm: Cross sectional Algebra}.

Turning to the second, we begin with the following theorem. It, or
something akin to it, seems to have been known to Luminet \cite[Remark 4.14]{Lum97}.
However, no proof or reference was given. We are grateful to Zinovy
Reichstein for the formulation of the theorem and for allowing us
to include his proof here.
\begin{thm}
\label{thm:Reichstein}Suppose $d,n\geq2$ and for $k=1,2,\cdots,n-1$,
let $X_{k}$ be the set of all $(A_{1},A_{2},\cdots,A_{d})\in M_{n}(\mathbb{C})^{d}$
such that the $A_{i}$ have a common $k$-dimensional invariant subspace.
Then $X_{k}$ is an irreducible algebraic variety of dimension $dn^{2}-(d-1)k(n-k)$,
and 
\[
\cup_{k=1}^{n-1}X_{k}=M_{n}(\mathbb{C})^{d}\backslash\mathcal{V}(d,n).
\]
\end{thm}
\begin{proof}
Evidently, the union of the $X_{k}$ is $M_{n}(\mathbb{C})^{d}\backslash\mathcal{V}(d,n)$.
Let ${\rm Gr}(k,n)$ denote the Grassmannian consisting of all $k$-dimensional
subspaces of $\mathbb{C}^{n}$ and let
\[
Y_{k}=\{(A_{1},A_{2},\cdots,A_{d};W)\in M_{n}(\mathbb{C})^{d}\times{\rm Gr}(k,n)\mid A_{i}W\subseteq W,\,1\leq i\leq d\}.
\]
Clearly, $Y_{k}$ is an algebraic subvariety of $M_{n}(\mathbb{C})^{d}\times{\rm Gr}(k,n)$.
Let $\pi_{2k}:Y_{k}\to{\rm Gr}(k,n)$ be the projection onto the last
component. Then $\pi_{2k}$ is surjective, and its fibres are vector
spaces of block-upper triangular matrices (in appropriate bases),
with blocks of size $k$ and $n-k$. So the fibres are irreducible
varieties of the same dimension, viz., $d(n^{2}-k(n-k)).$ By the
fibre dimension theorem \cite[Theorem I.6.7, p.76]{Shafarevich_1_2007},
the $Y_{k}$ are irreducible and 
\[
\dim Y_{k}=d(n^{2}-k(n-k))+\dim{\rm Gr}(k,n)=dn^{2}-(d-1)k(n-k).
\]
Consider the map $\pi_{1k}:Y_{k}\to M_{n}(\mathbb{C})^{d}$ which
projects onto the first $d$ components. The image of $\pi_{1k}$
is $X_{k}$. Therefore, $X_{k}$ is irreducible. Further, the set
of $(A_{1},A_{2},\cdots,A_{d})\in X_{k}$ such that $A_{1}$ has distinct
eigenvalues is a Zariski open subset of $X_{k}$ and so $\dim X_{k}=\dim Y_{k}=dn^{2}-(d-1)k(n-k)$,
as claimed.
\end{proof}
If $(d,n)\neq(2,2)$, the complement of $\mathcal{V}(d,n)$ in $M_{n}(\mathbb{C})^{d}$
is the finite union of algebraic varieties of codimension $\geq2$
by Theorem \eqref{thm:Reichstein}. Consequently, by \cite[Theorem K.1]{Gunning_II_1990}
every function that is holomorphic on $\mathcal{V}(d,n)$ extends
uniquely to a function that is holomorphic on all of $M_{n}(\mathbb{C})^{d}$. 

Suppose, finally, $(d,n)=(2,2)$, and consider the commutator $[Z_{1},Z_{2}]$
in $\mathbb{G}_{0}(2,2)$. It is well known in some circles that $\mathcal{V}(2,2)=\{\mathfrak{z}=(Z_{1},Z_{2})\mid[Z_{1},Z_{2}]\,\mbox{is invertible}\}$.
Since we don't have an explicit reference for this, here is a simple
proof: One may assume, without loss of generality, that $Z_{1}$ is
in Jordan canonical form and that $Z_{1}$ either has distinct eigenvalues
or is the Jordan cell, $\begin{bmatrix}0 & 1\\
0 & 0
\end{bmatrix}$. If $Z_{1}$ has distinct eigenvalues, say $a$ and $c$, then we
may write
\[
[Z_{1},Z_{2}]=\left[\begin{bmatrix}a & 0\\
0 & c
\end{bmatrix},\begin{bmatrix}w & x\\
y & z
\end{bmatrix}\right]=\begin{bmatrix}0 & (a-c)x\\
(c-a)y & 0
\end{bmatrix}.
\]
If $Z_{1}=\begin{bmatrix}0 & 1\\
0 & 0
\end{bmatrix}$, then
\[
[Z_{1},Z_{2}]=\left[\begin{bmatrix}0 & 1\\
0 & 0
\end{bmatrix},\begin{bmatrix}w & x\\
y & z
\end{bmatrix}\right]=\begin{bmatrix}-y & w-z\\
0 & y
\end{bmatrix}.
\]
In either case, it is clear that $[Z_{1},Z_{2}]$ is invertible if
and only $Z_{1}$ and $Z_{2}$ have no common invariant subspace.
Thus $\det[Z_{1},Z_{2}]$ is a polynomial in $\mathbb{I}(2,2)$ whose
zero set is $M_{2}(\mathbb{C})^{2}\backslash\mathcal{V}(2,2)$. Thus
$f(\mathfrak{z}):=(\det[Z_{1},Z_{2}])^{-1}$is a holomorphic matrix
concomitant on $\mathcal{V}(2,2)$ that cannot be analytically extended
beyond $\mathcal{V}(2,2)$. Thus $\mathcal{V}(2,2)$ is a domain of
holomorphy in $M_{2}(\mathbb{C})^{2}$ and the proof of Theorem \ref{thm: Cross sectional Algebra}
is complete.

\section{Function Theory in Bundles\label{sec:Function-Theory-in-Bundles}}

Our first objective is to show that the center of $\mathbb{S}(\mathcal{D},\mathfrak{P};d,n)$
is $\mathbb{I}(\mathcal{D};d,n)$ independent of the reduction $\mathfrak{P}$.
Of course, $\mathbb{I}(\mathcal{D};d,n)$ is contained in the center.
The problem is the reverse inclusion. It is easy to see that every
element in the center of $\mathbb{S}(\mathcal{D},\mathfrak{P};d,n)$
is the restriction to $\partial\mathcal{D}$ of a continuous function
on $\overline{\mathcal{D}}$ that is holomorphic on $\mathcal{D}$,
but why must it be in $\mathbb{I}(\mathcal{D};d,n)$? The reason is
due, really, to Procesi who shows that the center of $\mathbb{S}_{0}(d,n)$
is $\mathbb{I}_{0}(d,n)$ \cite[Page 94]{Procesi1973}. 

First, note that the cross section $\varepsilon$ in $\Gamma_{c}(\partial\mathcal{D},\mathfrak{M}^{*}(\mathfrak{P};d,n))$
defined by the formula $\varepsilon([\mathfrak{z}]):=[\mathfrak{z},I_{n}]$,
where $I_{n}$ is the identity $n\times n$ matrix, is the identity
of $\Gamma_{c}(\partial\mathcal{D},\mathfrak{M}^{*}(\mathfrak{P};d,n))$.
Further, the center of $\Gamma_{c}(\partial\mathcal{D},\mathfrak{M}^{*}(\mathfrak{P};d,n))$,
$\mathfrak{Z}\Gamma_{c}(\partial\mathcal{D},\mathfrak{M}^{*}(\mathfrak{P};d,n))$,
is the set of all cross sections $\sigma$ of the form $\sigma([\mathfrak{z}])=[\mathfrak{z},c([\mathfrak{z}])I_{n}]$,
where $c:\partial\mathcal{D}\to\mathbb{C}$ is a continuous complex-valued
function. We shall usually write such a section as $c\cdot\varepsilon$,
and we shall identify $\mathfrak{Z}\Gamma_{c}(\partial\mathcal{D},\mathfrak{M}^{*}(\mathfrak{P};d,n))$
with $C(\partial\mathcal{D})$ through the isomorphism $c\to c\cdot\varepsilon$.
We shall write $\tau_{0}$ for the normalized trace on $M_{n}(\mathbb{C})$,
i.e., $\tau_{0}(I_{n})=1$, and we shall define $\tau:\mathfrak{M}^{*}(\mathfrak{P};d,n)\to\mathbb{C}$
by $\tau([\mathfrak{z},A]):=\tau_{0}(A)$. Then $\tau$ is a well-defined
continuous function on $\mathfrak{M}^{*}(\mathfrak{P};d,n)$. We now
define 
\[
T:\Gamma_{c}(\partial\mathcal{D},\mathfrak{M}^{*}(\mathfrak{P};d,n))\to\mathfrak{Z}\Gamma_{c}(\partial\mathcal{D},\mathfrak{M}^{*}(\mathfrak{P};d,n))
\]
 by the formula, 
\[
T(\sigma):=\tau\circ\sigma\cdot\varepsilon,\qquad\sigma\in\Gamma_{c}(\partial\mathcal{D},\mathfrak{M}^{*}(\mathfrak{P};d,n)).
\]
Then it is straightforward to verify that $T$ is a conditional expectation
from $\Gamma_{c}(\partial\mathcal{D},\mathfrak{M}^{*}(\mathfrak{P};d,n))$
onto $\mathfrak{Z}\Gamma_{c}(\partial\mathcal{D},\mathfrak{M}^{*}(\mathfrak{P};d,n))$
that also satisfies the equation 
\[
T(\Psi(\phi))([\mathfrak{z}])=\tau_{0}(\phi(\mathfrak{z}))\varepsilon([\mathfrak{z}]),\qquad\phi\in\mathbb{S}_{0}(d,n),\,\mathfrak{z}\in\mathcal{V}(d,n).
\]

\begin{thm}
\label{thm:Center} $T$ maps $\mathbb{S}(\mathcal{D},\mathfrak{P};d,n)$
onto $\mathbb{I}(\mathcal{D};d,n)$ and $\mathbb{I}(\mathcal{D};d,n)$
is the center of $\mathbb{S}(\mathcal{D},\mathfrak{P};d,n)$.\end{thm}
\begin{proof}
Since $T(\Psi(\phi))([\mathfrak{z}])=\tau_{0}(\phi(\mathfrak{z}))\varepsilon([\mathfrak{z}])$
for every $\phi\in\mathbb{S}_{0}(d,n)$ and since $\mathfrak{z}\to\tau_{0}(\phi(\mathfrak{z}))$
is a $G$-invariant polynomial function, the image of $T$ restricted
to $\mathbb{S}(\mathcal{D},\mathfrak{P};d,n)$ is contained in $\mathbb{I}(\mathcal{D};d,n)$.
If $\sigma$ is a section in the center of $\mathbb{S}(\mathcal{D},\mathfrak{P};d,n)$,
then $\sigma([\mathfrak{z}])$ lies in the center of the fibre of
$\mathfrak{M}^{*}(\mathfrak{P};d,n)$ over $[\mathfrak{z}]$, $\pi^{-1}([\mathfrak{z}])$,
which is (isomorphic to) $M_{n}(\mathbb{C})$. Thus $\sigma([\mathfrak{z}])$
is a multiple of $[\mathfrak{z},I_{n}]$. Hence, $\sigma\in\mathfrak{Z}\Gamma_{c}(\partial\mathcal{D},\mathfrak{M}^{*}(\mathfrak{P};d,n))$.
Since $\sigma\in\mathbb{S}(\mathcal{D},\mathfrak{P};d,n)$, there
is a sequence $\{\phi_{n}\}_{n\geq1}$ in $\mathbb{S}_{0}(d,n)$ such
that $\Psi(\phi_{n})\to\sigma$ in $\Gamma_{c}(\partial\mathcal{D},\mathfrak{M}^{*}(\mathfrak{P};d,n))$,
by definition of $\mathbb{S}(\mathcal{D},\mathfrak{P};d,n)$. But
then $T(\Psi(\phi_{n}))\to T(\sigma)=\sigma$ and each $T(\Psi(\phi_{n}))\in\mathbb{I}(\mathcal{D};d,n)$.
Thus $\sigma\in\mathbb{I}(\mathcal{D};d,n)$.\end{proof}
\begin{cor}
\label{cor:antisymmetric}$\mathbb{S}(\mathcal{D},\mathfrak{P};d,n)$
is a proper subalgebra of $\Gamma_{c}(\partial\mathcal{D},\mathfrak{M}^{*}(\mathfrak{P};d,n))$.
\end{cor}

\section{Boundary Representations \label{sec:Boundary-Representations}}

In this section, we prove Theorem \ref{thm:Boundary theorem}. It
rests on a simple observation of Kleski \cite[Remark 3.4]{Kleski2014},
which is a corollary of his deep Theorem 3.1. Recall Arveson's definition
of a peaking representation.
\begin{defn}
\label{def: peaking} \cite[Definition 7.1]{Arv2011} Suppose $A$
is a norm closed subalgebra of a unital $C^{*}$-algebra $B$ that
generates $B$ as a $C^{*}$-algebra and contains the unit of $B$.
An irreducible $C^{*}$-representation $\pi:B\to B(H_{\pi})$ is called
a \emph{peaking} \emph{representation for} $A$ if there is an integer
$n\geq1$ and an $n\times n$ matrix $(a_{ij})\in M_{n}(A)$ such
that 
\[
\Vert(\pi(a_{ij}))\Vert\gneq\Vert(\sigma(a_{ij}))\Vert
\]
for every irreducible representation $\sigma$ for $B$ that is not
unitarily equivalent to $\pi$. We also say that $\pi$ peaks at $(a_{ij})$.
\end{defn}
Arveson defines the notion of a peaking representation in the context
of operator systems, i.e., unital, closed, and self-adjoint subspaces
of $C^{*}$-algebras. However, thanks to \cite[Proposition 1.2.8]{Arv1969a},
if a representation is peaking in the sense of our Definition \ref{def: peaking}
it is a peaking representation with respect to the operator system
generated by $A$, i.e., the norm-closure of $A+A^{*}$.

In \cite[Thereom 3.1]{Kleski2014}, Kleski proves that if $(a_{ij})$
is any element in $M_{n}(A)$ then there is a boundary representation
$\pi_{0}$ of $B$ for $A$ such that 
\begin{equation}
\Vert(a_{ij})\Vert=\Vert(\pi_{0}(a_{ij}))\Vert.\label{eq:bdry_extreme}
\end{equation}
As Kleski observes in \cite[Remark 3.4]{Kleski2014}, this implies
that a peaking representation is a boundary representation. Indeed,
if $\pi$ is an irreducble representation of $B$ that peaks at $(a_{ij})$,
then we would have $\Vert(a_{ij})\Vert\geq\Vert(\pi(a_{ij}))\Vert\gneq\Vert(\pi_{0}(a_{ij}))\Vert$
if $\pi_{0}\nsim\pi$, which would contradict \eqref{eq:bdry_extreme}.
Thus $\pi\sim\pi_{0}$ and therefore $\pi$ is a boundary representation.

\begin{proof}[Proof of Theorem \ref{thm:Boundary theorem}] To apply
these remarks to the situation of Theorem \ref{thm:Boundary theorem}
is very easy. Our $B$ is $\Gamma_{c}(\partial\mathcal{D},\mathfrak{M}^{*}(\mathfrak{P};d,n))$
and our $A$ is $\mathbb{S}(\mathcal{D},\mathfrak{P};d,n)$. Our hypothesis
is that $\mathfrak{u}\in\partial_{e}\mathcal{D}$ - the extreme boundary
of $\mathcal{D}$. Since $Q_{0}(d,n)$ is metrizable, so is $\overline{\mathcal{D}}$.
Therefore $\mathfrak{u}$ is a peak point in the function algebra
sense \cite[Theorem 1.7.26]{Stout1971}, i.e., there is a function
$f\in\mathbb{I}(\mathcal{D};d,n)$ such that $f(\mathfrak{u})=1$,
but $\vert f(\mathfrak{v})\vert\lneq1$ for all $\mathfrak{v}\neq\mathfrak{u}$.
But then, we may simply view $f$ as a $1\times1$ matrix over $\mathbb{S}(\mathcal{D},\mathfrak{P};d,n)$
and conclude that $ev_{\mathfrak{u}}$ peaks at $f$. Hence $ev_{\mathfrak{u}}$
is a boundary representation of $\Gamma_{c}(\partial D,\mathfrak{P};d,n)$
for $\mathbb{S}(\mathcal{D},\mathfrak{P};d,n)$. \end{proof}

\begin{proof}[Proof of Corollary \ref{cor:Shilov_boundary_ideal}]
Any section $\sigma\in\Gamma_{c}(\partial\mathcal{D},\mathfrak{M}^{*}(\mathfrak{P};d,n))$
in the kernel of $ev_{\mathfrak{u}}$ vanishes at $\mathfrak{u}$.
So any section in $\cap_{\mathfrak{u}\in\partial_{e}\mathcal{D}}\ker(ev_{\mathfrak{u}})$
vanishes on $\partial_{e}\mathcal{D}$. Since $\partial_{e}\mathcal{D}$
is dense in $\partial\mathcal{D}$ \cite[Theorem I.7.24]{Stout1971},
any such section is the zero section. Therefore, by Theorem \ref{thm:Boundary theorem},
the intersection of the kernels of the boundary representations, which
is the Shilov boundary ideal, must be zero, i.e., $\Gamma_{c}(\partial\mathcal{D},\mathfrak{M}^{*}(\mathfrak{P};d,n))$
is the $C^{*}$-envelope of $\mathbb{S}(\partial\mathcal{D},\mathfrak{P};d,n)$.
\end{proof}

\section{Azumaya Algebras\label{sec:Azumaya-Algebras}}

The proof of Theorem \ref{thm:Azumaya} is an application of Procesi's
extension \cite[Theorem VIII.2.1]{Procesi1973} of Artin's theorem
that was discussed earlier. A $d$-variable \emph{central polynomial}
for the $n\times n$ matrices is a nonzero polynomial $p$ in the
center of $\mathbb{G}_{0}(d,n)$ \emph{that is without constant term.
}It is not evident, \emph{a priori}, that such polynomials exist.
However, they do - for every $d$ - thanks to the work of Formanek
\cite{Formanek_1972} and Razmyslov \cite{Razmyslov_1974}. Procesi's
theorem asserts (among many things) that if $R$ is a ring satisfying
the identities of the $n\times n$ matrices then $R$ is an Azumaya
algebra if and only if $R=F(R)R$ - the ideal generated by the Formanek
center, $F(R)$. The Formanek center, in turn, is the collection of
elements in $R$ obtained by evaluating all the central polynomials
for the $d$ generic $n\times n$ matrices for all $d$ at all $d$-tuples
of elements of $R$. Here, of course, when forming $F(R)$, we are
viewing a $d$-variable central polynomial $p$ as an element $\mathbb{C}\langle X_{1},X_{2},\cdots,X_{d}\rangle$.
Notice, too, that when $R=\mathbb{S}(\mathcal{D},\mathfrak{P};d,n)$,
then a $p\in\mathbb{G}_{0}(d,n)\subseteq\mathbb{S}(\mathcal{D},\mathfrak{P};d,n)$
may be identified with its evaluation at the $d$ coordinate functions
$\mathcal{Z}_{i}$, i.e., $p=p(\mathcal{Z}_{1},\mathcal{Z}_{2},\cdots,\mathcal{Z}_{d})$,
where, recall, $\mathcal{Z}_{i}(\mathfrak{z})=Z_{i}$, if $\mathfrak{z}=(Z_{1},Z_{2},\cdots,Z_{d})$.
We require the following special case of a lemma of Reichstein and
Vonessen \cite[Lemma 2.10]{RV2007}: For every $\mathfrak{z}\in\mathcal{V}(d,n)$
there is a $d$-variable central polynomial $p$ such that $p(\mathfrak{z})=I_{n}$. 

\begin{proof}[Proof of Theorem \ref{thm:Azumaya}]Now $\mathbb{S}(\mathcal{D},\mathfrak{P};d,n)$
certainly satisfies the identities of the $n\times n$ matrices and
our hypothesis on $\mathcal{D}$ is that $\overline{\mathcal{D}}$
is the maximal ideal space of $\mathbb{I}(\mathcal{D};d,n)$. Also,
our Theorem \ref{thm:Center} tells us that $\mathbb{I}(\mathcal{D};d,n)$
is the center of $\mathbb{S}(\mathcal{D},\mathfrak{P};d,n)$. So given
any point $\mathfrak{u}\in\overline{\mathcal{D}}$, we choose a $\mathfrak{z}$
in the bundle space $\mathcal{P}$ of $\mathfrak{P}$ such that $\pi(\mathfrak{z})=\mathfrak{u}$.
Then, using the Reichstein-Vonessen lemma, we choose a $d$-variable
central polynomial $p$ such that $p(\mathfrak{z})=I_{n}$. Since
a central polynomial certainly is invariant, we may view $p$ as a
function on $Q(d,n)$ that is $1$ at $\mathfrak{u}$. So, by the
compactness of $\overline{\mathcal{D}}$ we may choose a finite number
of central polynomials, $p_{1},p_{2},\cdots,p_{N}$, that have no
common zero on $\overline{\mathcal{D}}$. It follows that $p_{1}\mathbb{I}(\mathcal{D};d,n)+p_{2}\mathbb{I}(\mathcal{D};d,n)+\cdots+p_{N}\mathbb{I}(\mathcal{D};d,n)=\mathbb{I}(\mathcal{D};d,n)$
and, \emph{a fortiori}, that $F(\mathbb{S}(\mathcal{D},\mathfrak{P};d,n))\mathbb{S}(\mathcal{D},\mathfrak{P};d,n)=\mathbb{S}(\mathcal{D},\mathfrak{P};d,n)$.
Thus, $\mathbb{S}(\mathcal{D},\mathfrak{P};d,n)$ is an Azumaya algebra.
\end{proof}

\section{Concluding Remarks \label{sec:Concluding-Remarks}}

One may wonder about the extent of our results. How comprehensive
are the examples they cover? While we have formulated our analysis
in the context of the trace algebra of the algebra of generic matrices,
everything we have written goes over without significant changes to
the more general situation of what Reichstein and Vonessen call $n$-varieties.
\begin{defn}
\label{def:n-varieties}\cite[Definition 3.1]{RV2007} An \emph{n-variety
}is a $G$-invariant subset $X$ of $\mathcal{V}(d,n)$, for some
$d\geq2$, with the property that $X=\overline{X}\cap\mathcal{V}(d,n)$
where $\overline{X}$ denotes the Zariski closure of $X$ in $M_{n}(\mathbb{C})^{d}$.
\end{defn}
When passing from $\mathcal{V}(d,n)$ to an $n$-variety, one replaces
$\mathbb{G}_{0}(d,n)$ by $\mathbb{G}_{0}(d,n)/\mathcal{I}(X)$, where
$\mathcal{I}(X):=\{p\in\mathbb{G}_{0}(d,n)\mid p(\mathfrak{z})=0,\,\mathfrak{z}\in X\}$.
The quotient $\mathbb{G}_{0}(d,n)/\mathcal{I}(X)$ is a noncommutative
analogue of the coordinate ring of an algebraic variety and the thrust
of \cite{RV2007} is that noncommutative algebraic geometry should
take place in the context of $n$-varieties, their coordinate rings,
and associated noncommutative function fields. These latter are central
simple algebras and each can be written as the algebra of rational
matrix concomitants mapping $X$ into $M_{n}(\mathbb{C})$. Further,
by \cite[Lemma 8.1]{RV2007}, every irreducible algebraic variety
on which $G$ acts freely on a Zariski open set is birational to an
irreducible $n$-variety. Thus, with technical adjustments, the results
we have discussed make sense at this level. 

In another direction, which we are currently investigating, the results
of \cite{Muhly2013} suggest how to replace $PGL(n,\mathbb{C})$ with
certain more complicated reductive groups and formulate a function
theory on quiver varieties and other structures that can be built
from $C^{*}$-correspondences. 

The work of Craw, Raeburn and Taylor \cite{Craw1983} was also a source
of inspiration for us. They introduced the notion of a Banach Azumaya
algebra over a commutative Banach algebra. Their purpose was to use
the theory of Azumaya algebras to illuminate the topological properties
of the maximal ideal space of the commutative Banach algebra. However,
it seems difficult to identify naturally occurring Azumaya Banach
algebras ``in the wild''. Our results, coupled with their Proposition
2.6, show that such algebras arise quite naturally and quite frequently.

The specific problem which led us to the results we have presented
here stems from \cite{Muhly1998a} and \cite{Muhly2013}. In \cite{Muhly1998a}
we identified the completely contractive representations of the tensor
algebra of a $C^{*}$-correspondence. When the correspondence is specialized
to complex $d$-space $\mathbb{C}^{d}$, one finds that the completely
contractive $n$-dimensional representations of the tensor algebra,
$\mathcal{T}_{+}(\mathbb{C}^{d})$, are parametrized by the closed
``disc'', $\overline{\mathbb{D}(d,n)}$, where $\mathbb{D}(d,n)=\{\mathfrak{z}\in M_{n}(\mathbb{C})^{d}\mid\Vert\mathfrak{z}\mathfrak{z}^{*}\Vert<1\}$.
When viewed simply as a subset of the complex space $\mathbb{C}^{dn^{2}}$,
$\mathbb{D}(d,n)$ is a classical symmetric domain. If $\mathbb{G}(d,n)$
(resp. $\mathbb{S}(d,n)$) is the closure of $\mathbb{G}_{0}(d,n)$
(resp. $\mathbb{S}_{0}(d,n)$) in $C(\overline{\mathbb{D}(d,n)},M_{n}(\mathbb{C}))$,
then $\mathbb{G}(d,n)$ is precisely the sup-norm closure of the algebra
of functions on $\overline{\mathbb{D}(d,n)}$ that one obtains from
evaluating the elements of $\mathcal{T}_{+}(\mathbb{C}^{d})$ on $\overline{\mathbb{D}(d,n)}$.
(Note that Arveson \cite{Arv1998} showed that in general $\mathbb{G}(d,n)$
is strictly larger than the algebra of evaluations from $\mathcal{T}_{+}(\mathbb{C}^{d})$.)
The elements of $\mathbb{G}(d,n)$ and $\mathbb{S}(d,n)$ are continuous
$M_{n}(\mathbb{C})$-valued functions on $\overline{\mathbb{D}(d,n)}$
that are analytic on $\mathbb{D}(d,n)$ and for each $f\in\mathbb{S}(d,n)$,
the maximum of $\Vert f(\mathfrak{z})\Vert$, for $\mathfrak{z}\in\overline{\mathbb{D}(d,n)}$,
is taken on the Shilov boundary of $\mathbb{D}(d,n)$, $\partial_{e}\mathbb{D}(d,n)$.
The question which motivated this paper is ``What are the boundary
representations for $\mathbb{G}(d,n$) and $\mathbb{S}(d,n)$ and
what are the $C^{*}$-envelopes of these algebras?'' For this, we
need to know how to describe the $C^{*}$-algebras they generate.

The functions in $\mathbb{S}(d,n)$ are $K$-concomitants, i.e., $f(k^{-1}\mathfrak{z}k)=k^{-1}f(\mathfrak{z})k$
for all $\mathfrak{z}\in\overline{\mathbb{D}(d,n)}$ and all $k\in K$.
Therefore, the natural place to study them is on the quotient space
$\overline{\mathbb{D}(d,n)}/K$, which is a compact Hausdorff space
on which all the continuous $K$-concomitants, $C(\overline{\mathbb{D}(d,n)},M_{n}(\mathbb{C}))^{K}$,
naturally live. This algebra, in turn, is naturally isomorphic to
the $C^{*}$-algebra of continuous cross sections of a certain $C^{*}$-bundle
of finite dimensional $C^{*}$-algebras over $\overline{\mathbb{D}(d,n)}/K$,
by \cite[Lemma 2.2]{Echterhoff_and_Emerson_2011}. However, it is
\emph{not} a homogeneous $C^{*}$-algebra because $K$ does not act
freely on $\overline{\mathbb{D}(d,n)}$. There are some obvious candidates
for the boundary representations of $C(\overline{\mathbb{D}(d,n)},M_{n}(\mathbb{C}))^{K}$
for $\mathbb{S}(d,n)$, but we do not yet know how to check them.
Problems with isotropy prevent us from applying the ideas that we
have presented above. Nevertheless, the algebra $\mathbb{S}(d,n)$
seems to have a lot in common with the algebras $\mathbb{S}(\mathcal{D},\mathfrak{P};d,n)$
that we have discussed here. We focused on these first because we
could avoid difficulty with isotropy. The algebras $\mathbb{S}(\mathcal{D},\mathfrak{P};d,n)$
turn out to be quite interesting in their own right, however, and
they deserve further exploration.

\section*{acknowledgement}We are very grateful to Zinovy Reichstein
for enlightening and helpful correspondence concerning this paper. 

\bibliographystyle{plain}
\bibliography{/Users/Paul_Muhly/Dropbox/BP_Share/Master_Bib_File/Master20150822}
\end{document}